\documentclass[12pt]{amsart}
\usepackage{graphicx}
\usepackage{amssymb}
\usepackage{amscd}
\usepackage[all]{xy}
\usepackage[T1]{fontenc}
\usepackage{relsize}
\usepackage{tgheros}
\usepackage{geometry}
\usepackage{nameref}
\usepackage{hyperref}

\newtheorem{theorem}{Theorem}[section]
\newtheorem{proposition}{Proposition}[section]
\newtheorem{lemma}{Lemma}[section]

\newtheorem*{thmA}{Theorem A}
\newtheorem*{thmB}{Theorem B}
\newtheorem*{thmC}{Theorem C}

\newtheorem{remark}{Remark}[section]

\theoremstyle{definition}
\newtheorem{definition}{Definition}[section]

\usepackage{hyperref}

\hypersetup{
    colorlinks = true,
    linkcolor = blue
}
\subjclass{}
\keywords{}

\title[Open Set-valued map]{Spectral decomposition and entropy for open set-valued maps}

\email{cassio.morais@ufes.br}
\email{tulio.ferreira@ufv.br}
\email{rafael.pereira@ufop.edu.br}

\author{
Cássio H. Vieira Morais,
Túlio Vales\textsuperscript{*},
Rafael da C. Pereira
}
\address{UFES, Avenida Fernando Ferrari, 514, bairro Goiabeiras, Vitória - ES, CEP 29075-910.}
\address{UFV, Rodovia LMG-818, km 6, Florestal, MG, CEP 35.690-000.}
\address{UFOP, Rua Professor Paulo Magalhães Gomes, 122 - Bauxita, Ouro Preto - MG, CEP 35400-000.}

 \newcommand{\graf}{\operatorname{graph}}
  \newcommand{\orb}{\operatorname{Orb}}
 \newcommand{\htop}{h_{\text{top}}}
\newcommand\restric[2]{ {
	\left.\kern-\nulldelimiterspace 
	#1 
	\vphantom{\big|} 
	\right|_{#2}
	}}

\begin{document}



\begingroup
\renewcommand{\thefootnote}{*}
\footnotetext{Corresponding author: \texttt{tulio.ferreira@ufv.br}}
\endgroup
\begin{abstract}
In this paper, we study the dynamics of set-valued maps whose graphs are open and such that the image of each point is an open and connected set.
Building upon the work of P.~Duarte and M.~Torres, who introduced and analyzed the combinatorial structure of the final recurrent set, we incorporate the notions of transitivity and mixing, thereby bringing their framework into the spirit of Smale’s spectral decomposition theorem.
We also demonstrate that such maps exhibit infinite topological entropy, and we establish a connection between transitive open set-valued maps and transitive Anosov diffeomorphisms.
\end{abstract}

\maketitle

\section{Introduction}\label{sec.Intro}

One of the central problems in dynamical systems during the last century was the characterization of structurally stable systems.  
For many years, the so-called \emph{stability conjecture} stimulated intense research activity; see, for instance, \cite{peixoto1962structural} and \cite{smale1967differentiable} for two landmark results.  
The conjecture was ultimately proved by Mañé \cite{mane1987proof}.  
As a consequence of these and many other contributions, it is now well understood that structural stability is deeply connected to hyperbolicity and, in particular, to the shadowing property of $\varepsilon$-orbits. 

A discrete dynamical system is defined by a map $F:X \to X$ on a metric space. The orbits of $F$, which describe the evolution of points in $X$, are the sequences of iterates $(x_n)_{n \geq 0}$, where $x_n =F^n(x_0)$. An \emph{$\varepsilon$-orbit} is a sequence $(y_n)_{n \ge 0}$ satisfying 
$\operatorname{dist}(F(y_n), y_{n+1}) < \varepsilon$.  We say that the orbit $(x_n)$ $\delta$-shadows the $\varepsilon$-orbit $(y_n)$ if $\operatorname{dist}(x_n,y_n)< \delta$. These notions motivate the definition of the set-valued map $F_\varepsilon : X \to 2^X$, given by 
$F_\varepsilon(x) = B_\varepsilon(F(x))$,  
where $2^X$ denotes the collection of nonempty subsets of $X$, and $B_\varepsilon(x) \subset X$ is the open ball of radius $\varepsilon$ centered at $x$.  
Then $\varepsilon$-orbits of $F$ correspond to true orbits of $F_\varepsilon$.  
Under mild assumptions, each $F_\varepsilon$ is an example of an \emph{open set-valued map}, that is, a map $f : X \to 2^X$ whose graph $\operatorname{graph}(f) \subset X \times X$ is open and such that $f(x)$ is nonempty, open, and connected for every $x \in X$.  
When $X$ is compact, every open set-valued map $f$ contains $F_\varepsilon$ for some map $F : X \to X$ and $\varepsilon > 0$. 

A natural question is whether the study of open set-valued maps provides meaningful information about the underlying dynamics.  
Given dynamical information on $F_\varepsilon$, what can be inferred about $F$'s own behavior?  
To illustrate that such maps capture relevant dynamical properties, let us recall a classical open problem in hyperbolic dynamics: the \emph{Anosov transitivity conjecture}.  
It remains unknown whether every Anosov diffeomorphism is transitive (see \cite{smale1967differentiable}), although this is widely believed to be the case.  
The conjecture is true on infranilmanifolds \cite{manning1974there}, and so far, no examples of Anosov diffeomorphisms defined on manifolds that are not infranilmanifolds have been found.  
Recent advances toward this conjecture appear in \cite{BMelo15, Micena22, MicenaHertz22}, and in the forthcoming works \cite{BMelo25, Salgado25}.  
Interestingly, the analogous statement for flows is false \cite{FrankWilliams80}.

Regarding the transitivity of open set-valued maps, if $F$ is transitive, then $F_\varepsilon$ is transitive for all $\varepsilon > 0$.  
However, the converse is generally false.  
In this paper, we prove that the converse holds in the case of Anosov diffeomorphisms:

\begin{thmA}\hypertarget{teo.A}{}
Let \( f : X \to X \) be an Anosov diffeomorphism.  
If \( f_\varepsilon : X \to 2^X \) is transitive for every \( \varepsilon > 0 \), then \( f \) is transitive.
\end{thmA}

Several authors have investigated the dynamical properties of set-valued maps.  
The definition of transitivity used in Theorem~\hyperlink{teo.A}{A} follows that introduced in \cite{luo2020shadowing}.  
Other related studies include \cite{cordeiro2016continuum, duarte2006combinatorial, metzger2017topological, pilyugin2008shadowingcontractive, pilyugin2008shadowinghiperbolic}.  
In particular, \cite{metzger2017topological} addresses topological stability for set-valued maps using techniques such as the shadowing property and inverse limit constructions, while \cite{pilyugin2008shadowingcontractive, pilyugin2008shadowinghiperbolic} studies shadowing from a different viewpoint.  
It is worth emphasizing that, among all these contributions, only \cite{duarte2006combinatorial, duarte2015stability} specifically deals with open set-valued maps.

Another question we consider in this work concerns the structure of the limit sets of open set-valued maps, in analogy with Smale’s celebrated spectral decomposition theorem \cite{smale1967differentiable}.  
In its classical formulation, the theorem asserts that for an Axiom~A diffeomorphism $f$, the nonwandering set can be decomposed into finitely many invariant components, each of which is transitive.  
Later,  R. Bowen \cite{bowen1971periodic} refined this description by showing that each transitive component can itself be decomposed into finitely many subsets on which $f$ acts cyclically and in a mixing fashion.  
In a similar spirit, \cite{duarte2006combinatorial} introduced the notion of the set of \emph{final recurrent points} $\Omega_{\mathrm{final}}(f)$ for an open set-valued map and analyzed the combinatorial structure of its equivalence classes.  
We further develop this structure by studying the dynamics of $f$ restricted to each class, leading to our second main result.

\begin{thmB}\hypertarget{teo.B}{}
Let $f : X \to 2^X$ be an open set-valued map. 
Then $\Omega_{\mathrm{final}}(f)$ can be decomposed into a finite collection of pairwise disjoint open sets 
\[
\Lambda_1, \Lambda_2, \dots, \Lambda_n,
\]
and each $\Lambda_j$ can be further decomposed into finitely many pairwise disjoint, open, and connected sets
\[
\Sigma_1^j, \Sigma_2^j, \dots, \Sigma_{n_j}^j.
\]
This decomposition satisfies:
\begin{enumerate}
    \item $f(\Lambda_j) = \Lambda_j$ for each $j = 1, \dots, n$;
    \item $f(\Sigma_k^j) = \Sigma_{k+1}^j$ for $k = 1, \dots, n_j - 1$, and $f(\Sigma_{n_j}^j) = \Sigma_1^j$;
    \item the restriction $\restric{f}{\Lambda_j}$ is transitive for every $j = 1, \dots, n$;
    \item the restriction $\restric{f^{n_j}}{\Sigma_k^j}$ is mixing for every $j = 1, \dots, n$ and $k = 1, \dots, n_j$.
\end{enumerate}
\end{thmB}

A further topic of interest concerns the topological entropy of open set-valued maps.  
The definition we adopt here is completely analogous to the classical one for single-valued maps. 
This definition was used in \cite{cordeiro2016continuum}, where the authors studied the entropy of continuum-wise expansive set-valued maps and those with the pointwise specification property.  
In our context, since an open set-valued map $f$ contains some $F_\varepsilon$ associated with $F : X \to X$, i.e., $f$ encapsulates the dynamics of all maps sufficiently close to $F$, it is natural to expect that $f$ exhibits infinite entropy.  
This is indeed the case, as shown by our third main result.

\begin{thmC}\hypertarget{teo.C}{}
Let $X$ be a compact metric space without isolated points, and let $f : X \to 2^X$ be an open set-valued map. 
Then the topological entropy of $f$ is infinite.
\end{thmC}

\smallskip
The paper is organized as follows.  
Section~\ref{preliminares} contains the basic definitions and preliminary results concerning open set-valued maps.  
In Section~\ref{espectral}, we discuss the relationship between mixing and transitive open set-valued maps and establish Theorem~\hyperlink{teo.B}{B}.  
Section~\ref{entropia} is devoted to the proof of Theorem~\hyperlink{teo.C}{C}, while Section~\ref{anosov} contains the proof of Theorem~\hyperlink{teo.A}{A}.

\section{Preliminary} \label{preliminares}

In this section, we present several fundamental definitions and preliminary results that will be essential for the development of the paper. Some of these results were originally established in \cite{duarte2006combinatorial}, to which we refer for further details. We also discuss the notion of transitivity, a concept that plays a central role in the present work and was not treated in \cite{duarte2006combinatorial}.

Let $X$ be a compact metric space, and let $2^X$ denote the collection of all nonempty subsets of $X$. A set-valued map $f:X\to 2^X$ is an application such that for each $x \in X$, $f(x)$ is a subset of $X$. The graph of this application is  the set 
$$\mbox{graph}\, (f)=\{(x,f(x)): x\in X\} \subset X \times X.$$
\begin{definition} 
A set-valued map $f : X \to 2^X$ is said to be \emph{open} if its graph is an open subset of $X \times X$, and each value $f(x) \subset X$ is nonempty, open and connected. 
\end{definition}

A natural example of an open set-valued map arises from a continuous function by associating to each point a small open neighborhood of its image.  Specifically, let $g:X \to X$ be a map.  Given $\epsilon>0$, we denote by $g_\varepsilon$ the  set-valued map $g_\epsilon: X \to 2^X$ defined by
\[
g_\epsilon(x) := B_\epsilon(g(x)),
\]
where $B_\epsilon(g(x))$ denotes the open ball of radius $\epsilon$ centered at $g(x)$. If $g$ is continuous and $X$ is connected, then $g_\epsilon$ is an open set-valued map.

\begin{proposition}
    Let $f: X \to 2^X$ be an open set-valued map. If $y \in f^n(x)$ for some integer $n \geq 1$, then there exists $\delta  > 0$ such that $B_\delta(y) \subset f^n(\tilde x)$ for every $\tilde x \in B_\delta(x)$.
    \label{prop.abertos}
\end{proposition}

\begin{proof}
    Let $x \in X$ and $y \in f(x)$. Since $\graf (f)$ is open, there exist open sets $A \ni x$ and $B \ni y$ in $X$ such that  $A \times B \subset \graf(f)$.  This implies that $B \subset f(a)$ for every $a \in A$. By induction, the same property holds when $x \in X$ and $y \in f^n(x)$. Since $A$ and $B$ are open, there exists $\delta >0$ such that $B_\delta(x) \subset A$ and $B_\delta(y) \subset B$. This concludes the proof.
\end{proof}

\begin{proposition}\label{prop.grafico.epsilon}
Let $f : X \to 2^X$ be an open set-valued map.  
Then there exists a map $F : X \to X$ and a constant $\varepsilon_0 > 0$ such that 
\[
B_{\varepsilon_0}(\operatorname{graph}(F)) \subset \operatorname{graph}(f).
\]
\end{proposition}

\begin{proof}
For each $x \in X$, by Proposition~\ref{prop.abertos}, there exists a point $y_x \in f(x)$ and a number $\delta_x > 0$ such that 
$B_{\delta_x}(y_x) \subset f(y)$ for every $y \in B_{\delta_x}(x)$.  
The family $\{ B_{\delta_x}(x) \}_{x \in X}$ forms an open cover of $X$.  
By compactness, there exists a finite subcover $\{ B_{\delta_x}(x) \}_{x \in I}$.  
Define  $\varepsilon_0 = \min\{\, \delta_x : x \in I \,\}$, and construct a map $F : X \to X$ as follows: for each $z \in X$, choose an index $x \in I$ such that $z \in B_{\delta_x}(x)$ and set $F(z) = y_x$.  
Then, by construction,
\[
B_{\varepsilon_0}(\operatorname{graph}(F)) \subset \operatorname{graph}(f),
\]
as required.
\end{proof}

\begin{proposition}\label{prop.imagem.conexa}
Let $X$ be connected, and let $f : X \to 2^X$ be an open set-valued map. 
Then $f(X)$ is connected.
\end{proposition}

\begin{proof}
Suppose, by contradiction, that $f(X)$ is not connected. 
Then there exist disjoint nonempty open sets $A, B \subset X$ such that 
\[
A \cup B \supset f(X), 
\quad 
A \cap f(X) \neq \emptyset, 
\quad \text{and} \quad 
B \cap f(X) \neq \emptyset.
\]
Define
\[
\tilde{A} = \{\, x \in X : f(x) \subset A \,\}
\quad \text{and} \quad
\tilde{B} = \{\, x \in X : f(x) \subset B \,\}.
\]
We claim that for every $x \in X$, either $x \in \tilde{A}$ or $x \in \tilde{B}$. 
Indeed, if there exists $x \in X$ such that $f(x)$ meets both $A$ and $B$, 
then $A \cup B \supset f(x)$ with 
$A \cap f(x) \neq \emptyset$ and $B \cap f(x) \neq \emptyset$, 
contradicting the connectedness of $f(x)$. 
Hence $X = \tilde{A} \cup \tilde{B}$.

We next show that $\tilde{A}$ and $\tilde{B}$ are open. 
Let $a \in \tilde{A}$. 
Since $f(a) \subset A$ and $f(a)$ is open, 
Proposition~\ref{prop.abertos} ensures the existence of $\delta > 0$ such that 
$f(y) \cap f(a) \neq \emptyset$ for all $y \in B_\delta(a)$. 
Because $f(a) \subset A$, this implies $f(y) \subset A$ for all $y \in B_\delta(a)$, 
and thus $B_\delta(a) \subset \tilde{A}$. 
Therefore $\tilde{A}$ is open, and the same reasoning applies to $\tilde{B}$. 

We have found disjoint nonempty open sets $\tilde{A}, \tilde{B} \subset X$ such that 
$X = \tilde{A} \cup \tilde{B}$, 
contradicting the connectedness of $X$. 
Hence $f(X)$ must be connected.
\end{proof}

A trajectory (or orbit) of a set-valued map $f$ is a sequence $(x_n)_{n \in \mathbb{I}} \subset X$ satisfying
\[
x_{n+1} \in f(x_n) \quad \text{for all } n \in \mathbb{I}
\]
where $\mathbb I$  is some interval of $\mathbb N$. If the sequence $(x_n)_n$ has a finite length, we call the trajectory finite.

Given $x,y \in X$, we write $x \rightsquigarrow y$ if there exists a finite trajectory starting at $x$ and ending at $y$. More precisely, $x \rightsquigarrow y$ if there exists a finite sequence $(x_n)_n \subset X$ such that
\[
x_0 = x, \quad x_n = y, \quad \text{and} \quad x_{k+1} \in f(x_k) \text{ for all } k = 0,1,\dots,n-1.
\]

We now define two limit sets associated with a set-valued map:

\begin{definition}
Let $f: X \to 2^X$ be an open set-valued map. The \emph{recurrent set} and the \emph{final recurrent set} are, respectively,
$$\Omega(f)=\{x\in X: x \rightsquigarrow x\} \qquad \text{and} \qquad \Omega_{\mathrm{final}}(f)=\{x\in X: \mbox{ if } x \rightsquigarrow y, \mbox{ then } y\rightsquigarrow x\}.$$
\end{definition}

\begin{remark}
    For any open set-valued map $f$, one always has $\Omega_{\mathrm{final}}(f) \subset \Omega(f)$. Figure~\ref{omegasdiferentes} illustrates the graph of an open set-valued map on $X = [0,1]$ for which $\Omega_{\mathrm{final}}(f) \neq \Omega(f)$. In this example, the points $x$ and $y$ belong to $\Omega(f)$; however, $y \notin \Omega_{\mathrm{final}}(f)$ since $y \rightsquigarrow x$, but $x \not\rightsquigarrow y$.
\end{remark}

\begin{figure}[h]
    \centering
    \includegraphics[scale=1]{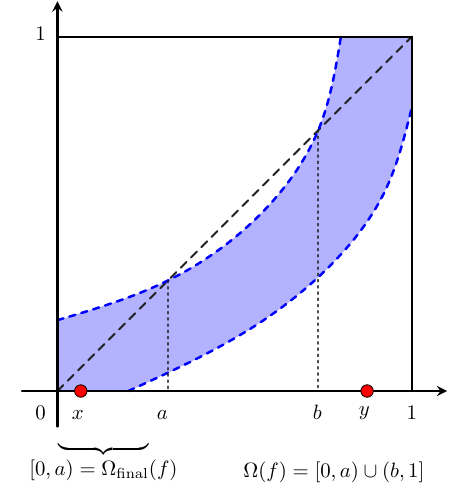}
    \caption{The graph of a open set-valued map such that $\Omega(f) \not= \Omega_{\textrm{final}}$.}
    \label{omegasdiferentes}
\end{figure}

The relation $\rightsquigarrow$ is an equivalence relation on $\Omega_{\mathrm{final}}(f)$, but it is not symmetric on $\Omega(f)$. To obtain an equivalence relation on $\Omega(f)$, we instead consider the relation $\leftrightsquigarrow$, defined by
\[
x \leftrightsquigarrow y \quad \text{if and only if} \quad x \rightsquigarrow y \ \text{and} \ y \rightsquigarrow x.
\]
We denote the corresponding equivalence classes by
\[
[x]_{\Omega} = \{\, y \in \Omega(f) : x \rightsquigarrow y \ \text{and} \ y \rightsquigarrow x \,\}
\quad \text{and} \quad
[x]_{\Omega_{\mathrm{final}}} = \{\, y \in \Omega_{\mathrm{final}}(f) : x \rightsquigarrow y \,\}.
\]

\begin{remark}
We refer to each $[x]_{\Omega_{\mathrm{final}}}$ as a \emph{final class}. There are only finitely many such classes, and each of them is an invariant set (see Theorem~\ref{teo.duarte}).
\end{remark}

\begin{proposition}
Let $f$ be an open set-valued map. Then both $[x]_\Omega$ and $\Omega(f)$ are open sets.
\label{omegaaberto}
\end{proposition}

\begin{proof}
Given $y \in [x]_\Omega$, there exists $n$ such that $y \in f^n(x)$, since $x \rightsquigarrow y$. Hence, we can find $\delta_1 >0$ as in Proposition~\ref{prop.abertos}. Similarly, there exists $m$ such that $x \in f^m(y)$, as $y \rightsquigarrow x$, and thus we can find $\delta_2$  satisfying  $f^m(\tilde y) \supset B_{\delta_2}(x)$ for every $\tilde y \in B_{\delta_2}(y)$.

Taking $\delta = \min\{ \delta_1, \delta_2\}$, for any $z \in B_\delta(y)$, we have
$$
\begin{cases}
x \in B_{\delta}(x) \subset B_{\delta_1}(x),\\[4pt]
z \in B_\delta(y) \subset B_{\delta_2}(y),
\end{cases}
\quad \Longrightarrow \quad
\begin{cases}
z \in B_{\delta_1}(y) \subset f^n(x),\\[4pt]
x \in B_{\delta_2}(x) \subset f^m(z),
\end{cases}
\quad \Longrightarrow \quad
\begin{cases}
x \rightsquigarrow z,\\[4pt]
z \rightsquigarrow x.
\end{cases}
$$
Therefore, $z \in [x]_\Omega$, and consequently $[x]_\Omega$ is open. 
Since $\Omega(f) = \bigcup\limits_{x \in \Omega} [x]_\Omega$, it follows that $\Omega(f)$ is also open.
\end{proof}

\begin{proposition}
Let $f$ be an open set-valued map. If $x \in \Omega_{\mathrm{final}}$, then $[x]_\Omega = [x]_{\Omega_{\mathrm{final}}}$. In particular, both $[x]_{\Omega_{\mathrm{final}}}$ and $\Omega_{\mathrm{final}}(f)$ are open sets.
\label{prop.classef.sao.abertas}
\end{proposition}

\begin{proof}
    Let $x \in \Omega_{\mathrm{final}}$. By definition, we have $[x]_{\Omega_{\mathrm{final}}} \subset [x]_\Omega$. On the other hand, given $y \in [x]_\Omega$, then $x \leftrightsquigarrow y$. If $y \rightsquigarrow z$, then $x \rightsquigarrow z$ by transitivity. Since $x \in \Omega_{\mathrm{final}}$, we also have $z \rightsquigarrow x$. By transitivity again, $z \rightsquigarrow y$, which implies that $ y \in \Omega_{\mathrm{final}}$. Hence, $[x]_{\Omega_{\mathrm{final}}} = [x]_\Omega$. 
\end{proof}

\begin{remark}
For any open set-valued map $f : X \to 2^X$, the sets $\Omega(f)$ and $\Omega_{\mathrm{final}}(f)$ are always nonempty. 
Indeed, to see that $\Omega(f) \neq \emptyset$, consider an orbit $(x_n)_{n \ge 0}$ of $f$. 
By the compactness of $X$, there exists a subsequence such that $x_n \to x$ for some $x \in X$.  
By Propositions~\ref{prop.abertos} and~\ref{prop.grafico.epsilon}, for sufficiently large $n$, we have $x_n \rightsquigarrow x$ and $x \rightsquigarrow x_n$, 
so $x \in \Omega(f)$.

The nonemptiness of $\Omega_{\mathrm{final}}(f)$ follows from Zorn’s Lemma.  
In fact, the relation $\rightsquigarrow$ induces a partial order on the equivalence classes of $\Omega(f)$, and the maximal elements with respect to this order correspond to final classes.  
See Lemma~5.5 in \cite{duarte2006combinatorial} for details.
\end{remark}

\begin{definition}
Let $f : X \to 2^X$ be a set-valued map. 
We say that $f$ is \emph{transitive} if, for any nonempty open sets $A, B \subset X$, there exists $n \in \mathbb{N}$ such that $f^n(A) \cap B \neq \emptyset$.
\end{definition}

The next result provides a useful reformulation of transitivity for open set-valued maps, showing that it suffices to consider single points instead of open sets.

\begin{proposition}
Let $f : X \to 2^X$ be an open and transitive set-valued map.  Then, for any $x \in X$ and any nonempty open set $B \subset X$, there exists $n \in \mathbb{N}$ such that 
$f^n(x) \cap B \neq \emptyset$. 
\label{prop.trans.pt}
\end{proposition}

\begin{proof}
    Since $f(x)$ is an open set, there exists, by the definition of transitivity, $m\in \mathbb{N}$ such that $f^m(f(x))\cap B \not= \emptyset$. Hence, if we let $n = m + 1$, we obtain $f^n(x) \cap B \neq \emptyset$, as desired.
\end{proof}

\begin{proposition}\label{restritoclassetransitivo}
Let $f : X \to 2^X$ be a set-valued map.  
Then the restriction $\restric{f}{[x]_{\Omega_{\mathrm{final}}(f)}}$ is transitive.
\end{proposition}

\begin{proof}
Let $U, V \subset [x]_{\Omega_{\mathrm{final}}(f)}$ be nonempty sets.  
Take $u \in U$ and $v \in V$.  
Since $u \rightsquigarrow v$, there exists $n \in \mathbb{N}$ such that $v \in f^n(u)$.  
Hence $f^n(U) \cap V \neq \emptyset$, which proves that the restriction of $f$ to $[x]_{\Omega_{\mathrm{final}}(f)}$ is transitive.
\end{proof}

\begin{proposition}\label{umaclasse}
Let $f : X \to 2^X$ be a transitive open set-valued map. 
Then $f(X) \subset \Omega(f)$. 
Moreover, $\Omega(f) = \Omega_{\mathrm{final}}(f)$, and this set consists of exactly one equivalence class.
\end{proposition}

\begin{proof}
We first claim that $f(X) \subset \Omega(f)$. 
Indeed, let $y \in f(X)$. 
Then there exists $x \in X$ such that $y \in f(x)$. 
By Proposition~\ref{prop.abertos}, there exists an open set $A \ni x$ such that $y \in f(a)$ for all $a \in A$. 
By the transitivity of $f$ (Proposition~\ref{prop.trans.pt}), there exists $n \in \mathbb{N}$ such that $f^n(y) \cap A \neq \emptyset$. Hence, there exists $a \in A$ with $a \in f^n(y)$, that is, $y \rightsquigarrow a$. Since $y \in f(a)$, we also have $a \rightsquigarrow y$, and therefore $y \rightsquigarrow y$, which shows that $y \in \Omega(f)$. This proves that $f(X) \subset \Omega(f)$.

Now, let $x \in \Omega(f)$ and $y \in X$ be such that $x \rightsquigarrow y$. 
By Proposition~\ref{omegaaberto}, the set $[x]_\Omega$ is open. 
Applying transitivity again (Proposition~\ref{prop.trans.pt}), there exists $z \in [x]_\Omega$ such that $y \rightsquigarrow z$. 
Since $z \in [x]_\Omega$ implies $z \leftrightsquigarrow x$, we obtain $y \rightsquigarrow x$. 
Hence $x \in \Omega_{\mathrm{final}}(f)$, and thus $\Omega(f) \subset \Omega_{\mathrm{final}}(f)$. 
As the reverse inclusion always holds, we conclude that $\Omega(f) = \Omega_{\mathrm{final}}(f)$.

Finally, we show that $\Omega(f)$ consists of a single equivalence class. 
Let $x \in \Omega(f)$ and $y \in \Omega(f)$ be arbitrary. 
By transitivity, there exists $z \in [x]_\Omega$ such that $y \rightsquigarrow z$. 
Since $z \leftrightsquigarrow x$, it follows that $y \rightsquigarrow x$. 
On the other hand, because $y \in \Omega_{\mathrm{final}}(f)$, we also have $x \rightsquigarrow y$. 
Therefore $y \leftrightsquigarrow x$, which shows that $\Omega(f) = [x]_\Omega$.
\end{proof}

\begin{remark}
For nontransitive open set-valued maps, it is possible to have more than one final class.  
Figure~\ref{duasclasses} illustrates the graph of an open set-valued map $f$, defined on $X = [0,1]$, for which $\Omega_{\mathrm{final}}(f)$ consists of exactly two distinct classes.
\end{remark}

\begin{figure}[htb]
    \centering
    \includegraphics[scale=0.83]{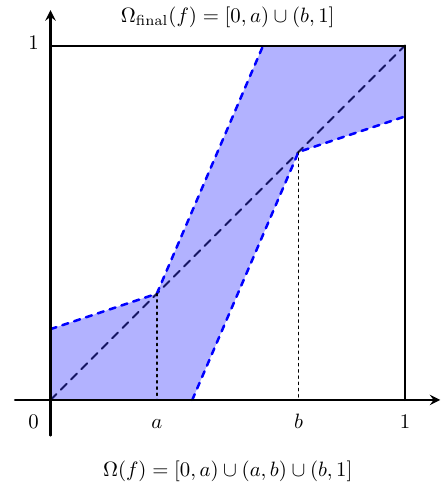}
    \caption{The graph of a open set-valued map with two final classes.}
    \label{duasclasses}
\end{figure}

\begin{proposition}\label{proposicaosobrejetor}
Let $f : X \to 2^X$ be a transitive open set-valued map. 
Then $\Omega_{\mathrm{final}}(f) = f(X)$, and this set is dense in $X$.
\end{proposition}

\begin{proof}
From Proposition~\ref{umaclasse}, we already know that $f(X) \subset \Omega_{\mathrm{final}}(f)$. 
Let $x \notin f(X)$   and $y \in f(x)$. Then $x \rightsquigarrow y$. However, since $x \notin \bigcup_{n=1}^\infty f^n(y) \subset f(X)$, it follows that $y \not\rightsquigarrow x$. 
Hence $x \notin \Omega_{\mathrm{final}}(f)$, which proves that $\Omega_{\mathrm{final}}(f) \subset f(X)$. 
Therefore, $\Omega_{\mathrm{final}}(f) = f(X)$.

To prove that $\Omega_{\mathrm{final}}(f)$ is dense in $X$, note that if $x \in \Omega_{\mathrm{final}}(f)$, then $f^n(x) \subset \Omega_{\mathrm{final}}(f)$ for all $n \in \mathbb{N}$. 
Hence
\[
\bigcup_{n=1}^\infty f^n(x) \subset \Omega_{\mathrm{final}}(f).
\]
Now let $V \subset X$ be a nonempty open set. 
By transitivity (Proposition~\ref{prop.trans.pt}), there exists $m \in \mathbb{N}$ such that 
$f^m(x) \cap V \neq \emptyset$. 
Consequently,
\[
\emptyset \neq f^m(x) \cap V \subset 
\bigcup_{n=1}^\infty f^n(x) \cap V \subset 
\Omega_{\mathrm{final}}(f) \cap V,
\]
showing that $\Omega_{\mathrm{final}}(f)$ intersects every nonempty open subset of $X$, and hence it is dense.
\end{proof}

\begin{proposition}\label{prop.imagem.de.transitivo.como.iterados.de.ponto}
Let $f : X \to 2^X$ be a transitive open set-valued map, and let $x \in X$. 
Then 
\[
f(X) = \bigcup_{n = 1}^{\infty} f^n(x).
\]
\end{proposition}

\begin{proof}
The inclusion 
\[
\bigcup_{n = 1}^{\infty} f^n(x) \subset f(X)
\]
is immediate. 

Conversely, let $y \in f(X)$. 
By Proposition~\ref{prop.trans.pt}, there exists $m \in \mathbb{N}$ such that 
\[
f^m(x) \cap \Omega_{\mathrm{final}}(f) \neq \emptyset.
\]
Hence, there exists $z \in \Omega_{\mathrm{final}}(f)$ with $x \rightsquigarrow z$. 
By Proposition~\ref{proposicaosobrejetor}, we have $y \in \Omega_{\mathrm{final}}(f)$, 
and by Proposition~\ref{umaclasse}, all points in $\Omega_{\mathrm{final}}(f)$ belong to the same equivalence class. 
Thus $y \leftrightsquigarrow z$, which implies $x \rightsquigarrow y$. 
Therefore $y \in f^n(x)$ for some $n \ge 1$, and we conclude that
\[
f(X) \subset \bigcup_{n = 1}^{\infty} f^n(x).
\]
The result follows.
\end{proof}

\section{Spectral decomposition for open set-valued maps} \label{espectral}

In this section, we revisit and extend the structural framework developed by P.~Duarte and M.~Torres~\cite{duarte2006combinatorial}. 
Our purpose is to incorporate the notions of transitivity and mixing into this setting, 
thereby complementing their description of the dynamics of open set-valued maps.
This will lead to a spectral decomposition result adapted to open set-valued maps.
As a preliminary step, we examine the relationship between transitivity and mixing in this context.

\begin{definition}
Let $f : X \to 2^X$ be a set-valued map. 
We say that $f$ is (topologically) \emph{mixing} if, for any nonempty open sets $A, B \subset X$, there exists $n_0 \in \mathbb{N}$ such that $f^n(A) \cap B \neq \emptyset $ for all $n \ge n_0$.
\end{definition}

In general, transitivity and mixing are distinct notions. 
The following examples illustrate this distinction and show the importance of the openness and connectedness assumptions that will appear later.

\medskip

\noindent
\textbf{Example 1.} 
Let $f : [0,1] \to 2^{[0,1]}$ be given by
\[
f(x) =
\begin{cases}
(\frac{1}{2}, 1], &  \text{if } x < \frac{1}{2}, \\[3pt]
[0,\frac{1}{2}), &  \text{if } x \ge \frac{1}{2}.
\end{cases}
\]
This map is transitive but not mixing. 
However, it is not an open set-valued map: indeed, $f([0,1])$ is not connected, as required by Proposition~\ref{prop.imagem.conexa}.

\medskip

\noindent
\textbf{Example 2.} 
Let $X = [0,1] \cup [3,4]$ and define $f : X \to 2^X$ by
\[
f(x) =
\begin{cases}
[0,1], &  \text{if } x \in [3,4] \\[3pt]
[3,4], &  \text{if } x \in [0,1].
\end{cases}
\]
This is a transitive open set-valued map that is not mixing. 
Here, the failure of mixing arises from the lack of connectedness of $X$. 
We will later see that, under connectedness, transitivity and mixing are indeed closely related.

\begin{proposition}\label{proposicao12}
Let $f : X \to 2^X$ be a transitive open set-valued map, and let $U \subset X$ be an open set such that $\overline{U} \subset f(X)$. 
Then there exists $n \ge 1$ such that $U \subset f^n(U)$.
\end{proposition}

\begin{proof}
Let $x \in \overline{U}$. 
By Proposition~\ref{proposicaosobrejetor}, we have $x \in f(X) = \Omega_{\mathrm{final}}(f)$, and hence $x \rightsquigarrow x$. 
By Proposition~\ref{prop.abertos}, there exist an open neighborhood $V_x \ni x$ and an integer $n_x \ge 1$ such that 
\[
V_x \subset f^{n_x}(y) \quad \text{for all } y \in V_x.
\]
The family $\{V_x\}_{x \in \overline{U}}$ forms an open cover of $\overline{U}$, which is compact since $X$ is compact. 
Therefore, there exist finitely many points $x_1, \ldots, x_k \in \overline{U}$ such that 
\[
\overline{U} \subset V_{x_1} \cup \cdots \cup V_{x_k}.
\]
Let $n = \operatorname{lcm}(n_{x_1}, \ldots, n_{x_k})$. 
For each $j = 1, \ldots, k$, we have
\[
V_{x_j} \subset f^{n_{x_j}}(U \cap V_{x_j}) 
\implies 
V_{x_j} \subset f^{kn_{x_j}}(U \cap V_{x_j}) 
\quad \text{for every } k \in \mathbb{N},
\]
and in particular for $k = n/n_{x_j}$ we obtain $V_{x_j} \subset f^{n}(U \cap V_{x_j})$.
Hence, $\overline{U} \subset f^n(U)$, and therefore $U \subset f^n(U)$, as desired.
\end{proof}

\begin{proposition}\label{proposicao16}
Let $X$ be a connected space, and let $f : X \to 2^X$ be a transitive open set-valued map. 
Suppose $U_0 \subset X$ is an open set such that 
$U_0 \subset f^n(U_0)$ for some $n \ge 1$. 
For each $k \ge 1$, define 
\[
U_k = f^k(U_0)
\quad \text{and} \quad
V_k = \bigcup_{j \ge 0} f^{jn}(U_k).
\]
Then $\bigcap_{j \ge 0} V_j \neq \emptyset$.
\end{proposition}

\begin{proof}
Because $f$ is transitive, we have $\Omega(f) = \Omega_{\mathrm{final}}(f) = f(X)$, and by Proposition~\ref{prop.imagem.de.transitivo.como.iterados.de.ponto}, 
the union of the forward iterates of $U_0$ covers $f(X)$:
\begin{equation}\label{eq.1}
f(X) = \bigcup_{j \ge 0} f^j(U_0).
\end{equation}

Since $U_n = f^n(U_0) \supset U_0$, repeated application of $f^n$ gives  $f^{kn}(U_j) \subset f^{(k+1)n}(U_j)  \quad \text{for all } k \in \mathbb{N}$.
Moreover, as $U_{j+n} = f^n(U_j) = f^n(f^j(U_0))$, we obtain $V_{j+n} = V_j$. 
Substituting this into~\eqref{eq.1}, it follows that 
\[
f(X) = V_1 \cup V_2 \cup \cdots \cup V_n.
\]
By Proposition~\ref{prop.imagem.conexa}, $f(X)$ is connected, so at least two of these sets intersect. 
Hence, there exist distinct $a, b \in \{1, \dots, n\}$ such that 
\[
V_a \cap V_b \neq \emptyset.
\]

By transitivity, the forward iterates of $V_a \cap V_b$ cover $f(X)$; that is,
\[
f(X) = \bigcup_{k \ge 0} f^k(V_a \cap V_b).
\]
Since $f(V_a \cap V_b) \subset V_{a+1} \cap V_{b+1}$, we deduce that
\[
f(X) = \bigcup_{0 \le i < j \le n} (V_i \cap V_j).
\]
Let $W_{ij} = V_i \cap V_j$. 
Because $f(X)$ is connected and covered by finitely many of the sets $W_{ij}$, at least two of them intersect; say 
$W_{ab} \cap W_{cd} \neq \emptyset$ for distinct pairs $(a,b) \neq (c,d)$. 
Consequently, there exist indices $\alpha, \beta, \gamma \in \{1, \dots, n\}$ such that 
\[
V_\alpha \cap V_\beta \cap V_\gamma \neq \emptyset.
\]
Repeating this argument finitely many times yields $\bigcap_{j \ge 0} V_j \neq \emptyset$, as claimed.
\end{proof}

\begin{proposition}\label{proposicao17}
Assume the same hypotheses as in Proposition~\ref{proposicao16}. 
Then $V_0 = f(X)$.
\end{proposition}

\begin{proof}
Let $x \in f(X)$, and let $v \in V_0 \cap V_1 \cap \cdots \cap V_{n-1}$ be a point given by Proposition~\ref{proposicao16}. 
For each $j \in \{0, 1, \dots, n-1\}$, there exists an integer $a_j \equiv j \pmod{n}$ such that $v \in U_{a_j}$. 
Let  $k = \max\{a_j : 0 \le j \le n-1\}$. Then 
\[
v \in U_k \cap U_{k+1} \cap \cdots \cap U_{k+n-1}.
\]

Since $x \in f(X)$, Proposition~\ref{prop.imagem.de.transitivo.como.iterados.de.ponto} guarantees the existence of some integer $m \ge 1$ such that $x \in f^m(v)$. 
Hence 
\[
x \in U_{k+m} \cap U_{k+m+1} \cap \cdots \cap U_{k+m+n-1}.
\]
Among the indices $\{k+m, k+m+1, \dots, k+m+n-1\}$, there is one multiple of $n$; denote it by $jn$. 
Therefore $x \in U_{jn} \subset V_0$, and the result follows.
\end{proof}

\begin{theorem}\label{transitivomixing}
Let $f : X \to 2^X$ be an open set-valued map, where $X$ is compact and connected. 
Then the following statements are equivalent:
\begin{enumerate}
    \item $f$ is transitive;
    \item $f$ is mixing.
\end{enumerate}
\end{theorem}

\begin{proof} The implication $(2) \Rightarrow (1)$ follows directly from the definitions, so it suffices to prove $(1) \Rightarrow (2)$.
Let $U, V \subset X$ be nonempty open sets. 
Choose an open set $U_0 \subset U$ such that $\overline{U_0} \subset U \cap f(X)$. 
By Proposition~\ref{proposicao12}, there exists $n \in \mathbb{N}$ with 
$U_0 \subset f^n(U_0)$.
Define $U_k = f^k(U_0)$ for $k \ge 1$. 
By Proposition~\ref{proposicao17}, we have
$$
f(X) = \bigcup_{k = 0}^{\infty} f^{kn}(U_0) 
     = \bigcup_{k = 0}^{\infty} f^{kn}(U_1)
     = \cdots 
     = \bigcup_{k = 0}^{\infty} f^{kn}(U_{n-1}).
$$

Since  $f(X)$ is dense (Proposition \ref{proposicaosobrejetor}), for each $j \in \{0, 1, \dots, n-1\}$ there exists $n_j \in \mathbb{N}$ such that
\[
f^{n_j n}(U_j) \cap V \neq \emptyset.
\]
Moreover, because $U_0 \subset f^n(U_0)$, it follows that $f^{kn}(U_j) \cap V \neq \emptyset$ for all $k \ge n_j$.

Let  $N = \max\{n_j : 0 \le j \le n-1\}$. If $m > Nn + (n-1)$, then there exist $j \in \{0, 1, \dots, n-1\}$ and $k \in \mathbb{N}$ such that $m - j = kn$ with $k \ge n_j$. 
Hence,
\[
f^m(U_0) \cap V  = f^{m-j}(U_j) \cap V  = f^{kn}(U_j) \cap V\neq \emptyset.
\]
Since $U_0 \subset U$, we have $f^m(U) \cap V \neq \emptyset$ for all $m$ sufficiently large, showing that $f$ is mixing.
\end{proof}

\begin{theorem}\label{teo.mixing2}
Let $f : X \to 2^X$ be an open set-valued map, and let $C \subset X$ be a nonempty connected open set such that $f(C) \subset C$. 
If the restriction $\restric{f}{C} : C \to 2^C$ is transitive, then it is mixing.
\end{theorem}

\begin{proof}
We first observe that if $x \in \partial C$, then $f(x) \subset C$. 
Indeed, suppose otherwise that $C \setminus f(x) \neq \emptyset$. 
By Proposition~\ref{prop.abertos}, for all $y$ sufficiently close to $x$, we would have $f(y) \not\subset C$, contradicting the assumption $f(C) \subset C$. 
Hence $f(x) \subset C$ for all $x \in \partial C$.

It follows that we may consider the restriction 
\[
\restric{f}{\overline{C}} : \overline{C} \to 2^{\overline{C}}.
\]
Since $\restric{f}{C}$ is transitive, so is $\restric{f}{\overline{C}}$. 
Because $X$ is compact, $\overline{C}$ is compact as well. 
Therefore, by Theorem~\ref{transitivomixing}, $\restric{f}{\overline{C}}$ is mixing. 
This implies that $\restric{f}{C}$ is also mixing.
\end{proof}

The following theorem brings together Propositions~5.1 and~5.3 and Theorem~5.1 from~\cite{duarte2006combinatorial}, 
to which we refer for complete proofs and further discussion. 
For convenience, we recall here the notation used in their statement.

\begin{definition}
Let $f : X \to 2^X$ be an open set-valued map. 
The collection of all final classes $[x]_{\Omega_{\mathrm{final}}}$ contained in $\Omega_{\mathrm{final}}(f)$ is denoted by $\Lambda^\Omega_{\mathrm{final}}(f)$. 
A connected component of a final class is called a \emph{final component}, and the set of all such components is denoted by $\Sigma^\Omega_{\mathrm{final}}(f)$.
\end{definition}

\begin{theorem}\label{teo.duarte}
Let $f : X \to 2^X$ be an open set-valued map. 
Then, the set of final classes $\Lambda^\Omega_{\mathrm{final}}(f)$ is finite, and each class $A \in \Lambda^\Omega_{\mathrm{final}}(f)$ satisfies $f(A) = A$. 
Moreover, the set of final components $\Sigma^\Omega_{\mathrm{final}}(f)$ is finite, and for every $C \in \Sigma^\Omega_{\mathrm{final}}(f)$ we have $f(C) \in \Sigma^\Omega_{\mathrm{final}}(f)$. 
The mapping
\[
\pi_f : \Sigma^\Omega_{\mathrm{final}}(f) \longrightarrow \Sigma^\Omega_{\mathrm{final}}(f), 
\qquad \pi_f(C) = f(C),
\]
is a bijection that cyclically permutes the components of each final class.
\end{theorem}

The preceding theorem describes the finite cyclic structure of the final components of an open set-valued map. 
Building upon this result  and Theorem \ref{transitivomixing}, we obtain a decomposition of the final recurrent set into finitely many transitive and mixing subsystems, 
which play the role of a \emph{spectral decomposition} for $f$.

\begin{thmB}[Spectral Decomposition]
Let $f : X \to 2^X$ be an open set-valued map. 
Then $\Omega_{\mathrm{final}}(f)$ can be decomposed into a finite collection of pairwise disjoint open sets 
\[
\Lambda_1, \Lambda_2, \dots, \Lambda_n,
\]
and each $\Lambda_j$ can be further decomposed into finitely many pairwise disjoint, open, and connected sets
\[
\Sigma_1^j, \Sigma_2^j, \dots, \Sigma_{n_j}^j.
\]
This decomposition satisfies:
\begin{enumerate}
    \item $f(\Lambda_j) = \Lambda_j$ for each $j = 1, \dots, n$;
    \item $f(\Sigma_k^j) = \Sigma_{k+1}^j$ for $k = 1, \dots, n_j - 1$, and $f(\Sigma_{n_j}^j) = \Sigma_1^j$;
    \item the restriction $\restric{f}{\Lambda_j}$ is transitive for every $j = 1, \dots, n$;
    \item the restriction $\restric{f^{n_j}}{\Sigma_k^j}$ is mixing for every $j = 1, \dots, n$ and $k = 1, \dots, n_j$.
\end{enumerate}
\end{thmB}

\begin{proof}
The existence of the above decomposition and properties~(1)–(2) follow directly from Theorem~\ref{teo.duarte}, 
where each $\Lambda_j$ corresponds to a final class of $\Omega_{\mathrm{final}}(f)$, 
and each $\Sigma_k^j$ corresponds to a final component within that class. 
Property~(3) is a consequence of Proposition~\ref{restritoclassetransitivo}. 

To establish property~(4), note that the transitivity of $\restric{f}{\Lambda_j}$ together with property~(2) 
implies that $\restric{f^{n_j}}{\Sigma_k^j}$ is transitive. 
Then, by Theorem~\ref{teo.mixing2}, each $\restric{f^{n_j}}{\Sigma_k^j}$ is mixing.
\end{proof}

\section{Topological entropy for open set-valued maps} \label{entropia}
In this section, we study the topological entropy of open set-valued maps. The notion we adopt is parallel to the classical definition of topological entropy for single-valued maps and was developed in \cite{cordeiro2016continuum}.
For completeness, we briefly recall the basic definitions and refer to these references for additional details.

Throughout, let $X$ be a compact metric space, and let $f : X \to 2^X$ be an open set-valued map.  We denote by $\orb(f)$ the set of all forward trajectories of $f$, that is,
\[
\orb(f)
= \Big\{ (x_i)_{i \ge 1} : x_{i+1} \in f(x_i) \text{ for all } i \ge 1 \Big\}.
\]
For each $n \in \mathbb{N}$, we denote by $\orb_n(f)$ the set of all trajectories of length $n$:
\[
\orb_n(f)
= \Big\{ (x_1,\dots,x_{n}) : x_{i+1} \in f(x_i) \text{ for } i = 1,\dots,n-1 \Big\}.
\]

\begin{definition}
A set $E \subset \orb_n(f)$ is called \emph{$(n,\varepsilon)$-separated} if, for any two distinct trajectories 
\[
(x_1,\dots,x_{n}),\ (y_1,\dots,y_{n}) \in E,
\]
there exists $1 \le i \le n$ such that $\operatorname{dist}(x_i,y_i) \ge \varepsilon$. 
We denote by $s_n(\varepsilon)$ the maximal cardinality of an $(n,\varepsilon)$-separated subset of $\orb_n(f)$.
\end{definition}

\begin{definition}
The \emph{topological entropy} of $f : X \to 2^X$ is defined by
\[
h_{\mathrm{top}}(f)
= \lim_{\varepsilon \to 0} \ \limsup_{n \to \infty} \frac{1}{n} \log s_n(\varepsilon).
\]
\end{definition}

\begin{remark}
In complete analogy with the classical case, the topological entropy can also be defined in terms of spanning sets. 
A set $E \subset \orb_n(f)$ is said to be \emph{$(n,\varepsilon)$-spanning} if for every 
\[
(x_1,\dots,x_{n}) \in \orb_n(f)
\]
there exists $(y_1,\dots,y_{n}) \in E$ such that $\operatorname{dist}(x_i,y_i) < \varepsilon$ for all $1 \le i \le n$. 
Let $r_n(\varepsilon)$ denote the minimal cardinality of an $(n,\varepsilon)$-spanning set. 
Then one also has
\[
h_{\mathrm{top}}(f)
= \lim_{\varepsilon \to 0} \ \limsup_{n \to \infty} \frac{1}{n} \log r_n(\varepsilon),
\]
which follows from the standard inequalities
\[
r_n(\varepsilon) \ \le\ s_n(\varepsilon) \ \le\ r_n(\varepsilon/2).
\]

Hence, we have that \[
h_{\mathrm{top}}(f)
= \lim_{\varepsilon \to 0} \ \limsup_{n \to \infty} \frac{1}{n} \log s_n(\varepsilon) = \lim_{\varepsilon \to 0} \ \limsup_{n \to \infty} \frac{1}{n} \log r_n(\varepsilon).
\]

\end{remark}


The next examples show that an open set-valued map may have zero or positive topological entropy.

First, let $X$ be a finite set and define $f : X \to 2^X$ by $f(x) = \{x\}$ for all $x \in X$. 
In this case every orbit is constant, so for $\varepsilon > 0$ sufficiently small the quantity $s_n(\varepsilon)$ is equal to  the cardinality of $X$ for all $n$. 
Hence $\htop(f) = 0$.

We can also obtain positive entropy in a very simple way. 
Let $X = \{a,b\}$ with $\operatorname{dist}(a,b) = \varepsilon_0$ and define
$f(a) = f(b) = \{a,b\}$.
Then, for every $\varepsilon < \varepsilon_0$, each sequence $(x_0,\dots,x_{n-1}) \in \{a,b\}^n$ is an admissible orbit of $f$, and any two distinct such sequences are $(n,\varepsilon)$-separated. 
Therefore $s_n(\varepsilon) = 2^n$ and $\htop(f) = \log 2$.

A common feature of both examples is that $f(x)$ is finite for every $x \in X$. 
In fact, if $f$ is an open set-valued map such that $f(x)$ is infinite for all points, then $\htop(f) = \infty$. 
This is precisely what happens, for example, when $X$ has no isolated points, and it is the content of Theorem \hyperlink{teoC}{C}.
The key idea is that one can select sufficiently many well-separated points inside each set $f(x)$, as made precise in Proposition~\ref{prop:preparacao.entropia}; 
Lemmas~\ref{lem.1} and~\ref{lem.2} provide the technical preparation for that result.

\begin{lemma}\label{lem.1}
Let $M$ be a connected metric space, and let $a,b \in M$ be points satisfying $\operatorname{dist}(a,b) \ge 2\varepsilon_0$ for some $\varepsilon_0 > 0$.  
Then, for every $x_0 \in M$ and every $\varepsilon \in [0,\varepsilon_0]$,  
the boundary 
\[
\partial B_\varepsilon(x_0) = \{\, x \in M : \operatorname{dist}(x, x_0) = \varepsilon \,\}
\]
is nonempty.
\end{lemma}

\begin{proof}
Fix $x_0 \in M$ and $\varepsilon \in [0,\varepsilon_0]$.  
The ball $B_\varepsilon(x_0)$ cannot contain both $a$ and $b$, so the set  
$\{\, x \in M : \operatorname{dist}(x,x_0) \ge \varepsilon \,\}$  
is nonempty.  
If $\partial B_\varepsilon(x_0)$ were empty, then
\[
M = B_\varepsilon(x_0) \cup \{\, x \in M : \operatorname{dist}(x,x_0) > \varepsilon \,\},
\]
which contradicts the connectedness of $M$.  
Hence $\partial B_\varepsilon(x_0) \neq \emptyset$.
\end{proof}

\begin{lemma}\label{lem.2}
Let $M$ be a compact metric space without isolated points.  
Then, for every $\varepsilon > 0$, there exists $r > 0$ such that, for every $x_0 \in M$, the set
\[
\{\, x \in M : r < \operatorname{dist}(x,x_0) < \varepsilon \,\}
\]
is nonempty.
\end{lemma}

\begin{proof}
Fix $\varepsilon > 0$.  
For each $x \in M$, the ball $B_\varepsilon(x)$ contains infinitely many points.  
Choose any $y \in B_\varepsilon(x)$ with $y \ne x$, and define
\[
r_x < \min\!\left\{\, \varepsilon - \operatorname{dist}(x,y), \, \frac{\operatorname{dist}(x,y)}{2} \,\right\}.
\]
Then, for every $x_0 \in B_{r_x}(x)$,
\[
r_x 
< \operatorname{dist}(x,y) - \operatorname{dist}(x,x_0)
\le \operatorname{dist}(x_0,y)
\le \operatorname{dist}(x,y) + \operatorname{dist}(x,x_0)
< \varepsilon.
\]
Hence $\operatorname{dist}(x_0,y) \in (r_x, \varepsilon)$, which proves that the set above is nonempty for every $x_0 \in B_{r_x}(x)$.  
The family $\{ B_{r_x}(x) \}_{x \in M}$ forms an open cover of $M$.  
By compactness, there exists a finite subcover $\{ B_{r_x}(x) \}_{x \in I}$.  
Setting $r = \min\{ r_x : x \in I \}$ yields the desired result.
\end{proof}

\begin{proposition}\label{prop:preparacao.entropia}
Let $X$ be a connected metric space without isolated points, and let $f : X \to 2^X$ be an open set-valued map.  
Given $m \in \mathbb{N}$, there exists $\varepsilon_0 > 0$ such that, for every $x \in X$, one can find points $x_1, \dots, x_m \in f(x)$ satisfying 
\[
\operatorname{dist}(x_i, x_j) > \varepsilon_0 
\quad \text{for all } 1 \le i < j \le m.
\]
\end{proposition}

\begin{proof}
Let $\varepsilon_1 > 0$ be as given by Proposition~\ref{prop.grafico.epsilon}.  
Applying Lemma~\ref{lem.2} to $\varepsilon = \varepsilon_1$ (with $M = X$), let $r > 0$ be the corresponding constant.  
This guarantees that $f(x)$ contains at least two points whose distance is greater than or equal to $r$.  
Since $f(x)$ is connected, Lemma~\ref{lem.1} can be applied within $f(x)$.

Fix $m \in \mathbb{N}$ and define
\[
\varepsilon_2 = \min\!\left\{\varepsilon_1, \frac{r}{2}\right\},
\qquad 
\text{and choose } \varepsilon_0 < \frac{\varepsilon_2}{m+1}.
\]
For each $x \in X$, Proposition~\ref{prop.grafico.epsilon} ensures the existence of a point $x_0 \in X$ such that 
\[
B_{\varepsilon_2}(x_0) \subset f(x).
\]
By Lemma~\ref{lem.1}, for each $j = 1, \dots, m$ there exists $x_j \in B_{\varepsilon_2}(x_0)$ such that 
\[
\operatorname{dist}(x_0, x_j) = \frac{j\,\varepsilon_2}{m+1}.
\]
Using the triangle inequality, for all $1 \le i < j \le m$ we have
\[
\operatorname{dist}(x_i, x_j)
\ge \operatorname{dist}(x_0, x_j) - \operatorname{dist}(x_0, x_i)
= \frac{(j-i)\,\varepsilon_2}{m+1}
> \varepsilon_0.
\]
This completes the proof.
\end{proof}

\begin{thmC} \hypertarget{teoC}{}
Let $X$ be a compact metric space without isolated points, and let $f : X \to 2^X$ be an open set-valued map. 
Then the topological entropy of $f$ is infinite.
\end{thmC}

\begin{proof}
Fix $m \in \mathbb{N}$. 
We will show that, for all sufficiently small $\varepsilon > 0$, one has $s_n(\varepsilon) \ge m^n$ for every $n \in \mathbb{N}$, 
where $s_n(\varepsilon)$ denotes the maximal cardinality of an $(n,\varepsilon)$-separated set for $f$. 
This will imply that $h_{\mathrm{top}}(f) = \infty$.

Let $\varepsilon_0 > 0$ be as given by Proposition~\ref{prop:preparacao.entropia}, and fix $n \in \mathbb{N}$. 
For each sequence $(r_1, \dots, r_n)$ with $1 \le r_j \le m$, we will construct an orbit of $f$. 
There are $m^n$ such sequences in total.

Choose any point $x \in X$. 
By Proposition~\ref{prop:preparacao.entropia}, we can find $x_1, \dots, x_m \in f(x)$ such that 
$\operatorname{dist}(x_i, x_j) > \varepsilon_0$
for all $i \ne j$. 
If the first element of the sequence is $r_1$, we set the first point of the orbit to be $x_{r_1}$. 
Next, applying Proposition~\ref{prop:preparacao.entropia} to the point $x_{r_1}$, we obtain points 
$x_{r_1 1}, \dots, x_{r_1 m} \in f(x_{r_1})$
again mutually separated by more than $\varepsilon_0$. 
If $r_2$ is the second term of the sequence, we define the second point of the orbit as $x_{r_1 r_2}$. 
Proceeding inductively, we construct $m^n$ distinct orbits corresponding to all possible sequences $(r_1, \dots, r_n)$.

By construction, the set $E_n$ consisting of these $m^n$ points is $(n,\varepsilon_0)$-separated. 
Hence $s_n(\varepsilon_0) \ge m^n$. 
Since $s_n(\varepsilon)$ is nondecreasing as $\varepsilon \to 0$, we have 
$ s_n(\varepsilon) \ge m^n $ for all $\varepsilon < \varepsilon_0$. Therefore,
\[
h_{\mathrm{top}}(f)
= \lim_{\varepsilon \to 0} \limsup_{n \to \infty} \frac{1}{n} \log s_n(\varepsilon)
\ge \lim_{\varepsilon \to 0} \limsup_{n \to \infty} \frac{1}{n} \log m^n
= \log m.
\]
Since $m$ can be chosen arbitrarily large, it follows that 
$h_{\mathrm{top}}(f) = \infty$.
\end{proof}

\begin{remark}
The Theorem \hyperlink{teoC}{C} can be extended to a more general setting.  
Let $X$ be a compact metric space without isolated points, and let $f : X \to 2^X$ be a set-valued map such that each $f(x)$ is open.  
In this case, $f$ is not necessarily an open set-valued map, since we do not assume that the images $f(x)$ are connected.  
Nevertheless, $f$ still has infinite topological entropy, and the argument follows the same general idea as in the proof of Theorem \hyperlink{teoC}{C}. 
More precisely, Lemma~\ref{lem.1} cannot be applied here because it requires connectedness, but Lemma~\ref{lem.2} can be refined to compensate for this and to establish an analogue of Proposition~\ref{prop:preparacao.entropia}.
\end{remark}

\section{Transitivity of Anosov diffeomorphisms} \label{anosov}

Anosov diffeomorphisms are the paradigmatic examples of uniformly hyperbolic systems and form the cornerstone of smooth hyperbolic dynamics. Yet, since their introduction in the 1960s, it remains an open question whether every Anosov diffeomorphism is transitive. In what follows, we reinterpret this classical problem in the framework of open set-valued maps.


In this section, we consider \( X \) to be a compact and connected manifold.  
Recall that an \emph{Anosov diffeomorphism} is a diffeomorphism \( f : X \to X \) such that \( X \) itself is a hyperbolic set for \( f \).  
Also, as introduced in Section~2, given a continuous map \( f : X \to X \) and \( \varepsilon > 0 \), we denote by 
\[
f_\varepsilon : X \to 2^X, \qquad f_\varepsilon(x) = B_\varepsilon(f(x)),
\]
the open set-valued map that associates to each point \( x \) the open ball of radius \( \varepsilon \) centered at \( f(x) \).

In general, if \( f : X \to X \) is transitive, then \( f_\varepsilon \) is also transitive for every \( \varepsilon > 0 \).  
The converse, however, does not hold: it may happen that \( f_\varepsilon \) is transitive for every \( \varepsilon > 0 \) even though \( f \) itself is not.  
For instance, this is precisely the case when \( f \) is the identity map on \( X \).  
In what follows, we show that the converse does hold when \( f \) is an Anosov diffeomorphism.  
To do so, we recall some basic concepts related to shadowing in hyperbolic systems; see, for instance, \cite{brin2002introduction}.

\begin{definition}
Given a diffeomorphism \( f : X \to X \), a \emph{\(\delta\)-orbit} of \( f \) is a sequence \( (x_i)_{i \in I} \subset X \) such that 
\[
d(x_{i+1}, f(x_i)) < \delta,
\]
where \( I \) is an interval of integers.
\end{definition}

\begin{lemma}[Shadowing Lemma]
Let \( \Lambda \) be a hyperbolic set for \( f : U \to M \).  
Then, for every \( \varepsilon > 0 \), there exists \( \delta > 0 \) such that if \( (x_k) \) is a finite or infinite \(\delta\)-orbit of \( f \) and \( \operatorname{dist}(x_k, \Lambda) < \delta \) for all \( k \),  
then there exists \( x \in \Lambda \) such that 
\[
\operatorname{dist}(f^k(x), x_k) < \varepsilon
\quad \text{for all } k.
\]
\end{lemma}

We are now ready to prove the main result of this section.

\begin{thmA}
Let \( f : X \to X \) be an Anosov diffeomorphism.  
If \( f_\varepsilon : X \to 2^X \) is transitive for every \( \varepsilon > 0 \), then \( f \) is transitive.
\end{thmA}

\begin{proof}
First observe that if \( f_\delta \) is transitive, then there exists a dense \(\delta\)-orbit of \( f \).  
Indeed, let \( \{U_1, U_2, \dots\} \) be a countable basis of open sets for \( X \).  
By the transitivity of \( f_\delta \), given \( x_1 \in U_1 \), there exists \( k_1 \in \mathbb{N} \) such that 
\( f^{k_1}(x_1) \cap U_2 \neq \emptyset \).  
Choose \( x_2 \in f^{k_1}(x_1) \cap U_2 \).  
Repeating this process, we construct a sequence \( (x_n)_{n \ge 1} \) with \( x_n \in U_n \), forming a dense trajectory of the set-valued map \( f_\delta \).  
Notice that any trajectory of \( f_\delta \) is, by definition, a \(\delta\)-orbit of \( f \).  
Thus, for each \( \delta > 0 \), there exists a dense \(\delta\)-orbit of \( f \).

By the Shadowing Lemma, for every \( \varepsilon > 0 \) there exists \( \delta > 0 \) such that every \(\delta\)-orbit of \( f \) is \(\frac{\varepsilon}{2}\)-shadowed by a true orbit of \( f \).  
Hence, there exists a point \( x_\varepsilon \in X \) such that the sequence \( (f^n(x_\varepsilon))_{n \ge 0} \) \(\frac{\varepsilon}{2}\)-shadows the dense \(\delta\)-orbit \( (x_n)_{n \ge 0} \) constructed above.  
Since \( (x_n) \) is dense, the orbit of \( x_\varepsilon \) intersects every \(\varepsilon\)-neighborhood of every point of \( X \) infinitely many times.

To prove that \( f \) is transitive, let \( U, V \subset X \) be nonempty open sets.  
Choose points \( u, v \in X \) and \( \varepsilon > 0 \) such that \( B_\varepsilon(u) \subset U \) and \( B_\varepsilon(v) \subset V \).  
With $x_\varepsilon$ constructed above, there exist integers \( m, n \in \mathbb{N} \), with \( n > m \), such that 
\( f^m(x_\varepsilon) \in U \) and \( f^n(x_\varepsilon) \in V \).  
Therefore, $f^{n-m}(U) \cap V \neq \emptyset$ and we conclude that \( f \) is transitive.
\end{proof}

\section{Acknowledgment}
The authors would like to thank Vitor Araújo (IME-UFBA) and Bernardo Carvalho (LNCC) for their introduction and suggestions regarding the topics covered in this work.
\bibliographystyle{siam}
\bibliography{bib}

@article{bowen1971periodic,
  title={Periodic points and measures for Axiom A diffeomorphisms},
  author={Bowen, Rufus},
  journal={Transactions of the American Mathematical Society},
  volume={154},
  pages={377--397},
  year={1971},
  publisher={JSTOR}
}

@article{smale1967differentiable,
  title={Differentiable dynamical systems},
  author={Smale, Stephen},
  journal={Bulletin of the American mathematical Society},
  volume={73},
  number={6},
  pages={747--817},
  year={1967}
}

@article{luo2020shadowing,
  title={On the shadowing property and shadowable point of set-valued dynamical systems},
  author={Luo, Xiao Fang and Nie, Xiao Xiao and Yin, Jian Dong},
  journal={Acta Mathematica Sinica, English Series},
  volume={36},
  number={12},
  pages={1384--1394},
  year={2020},
  publisher={Springer}
}

@article{pilyugin2008shadowinghiperbolic,
  title={SHADOWING AND INVERSE SHADOWING
IN SET-VALUED DYNAMICAL SYSTEMS.
HYPERBOLIC CASE},
  author={Pilyugin, Sergei Yu and Rieger, Janosch},
  journal={Journal of the Juliusz Schauder Center},
  volume={32},
  pages={151--164},
  year={2008}
}

@article{pilyugin2008shadowingcontractive,
  title={Shadowing and inverse shadowing in set-valued dynamical systems. Contractive case},
  author={Pilyugin, Sergei Yu and Rieger, Janosch},  
  journal={Topological Methods in Nonlinear Analysis},
  volume={32},
  pages={139--149},  
  year={2008}
}

@article{metzger2017topological,
  title={Topological stability in set-valued dynamics},
  author={Metzger, Roger and Rojas, C Arnoldo Morales and Thieullen, Phillipe},
  journal={Discrete Contin. Dyn. Syst. Ser. B},
  volume={22},
  number={5},
  pages={1965--1975},
  year={2017}
}

@article{cordeiro2016continuum,
  title={Continuum-wise expansiveness and specification for set-valued functions and topological entropy},
  author={Cordeiro, Welington and Pac{\'\i}fico, Maria},
  journal={Proceedings of the American Mathematical Society},
  volume={144},
  number={10},
  pages={4261--4271},
  year={2016}
}

@book{brin2002introduction,
  title={Introduction to dynamical systems},
  author={Brin, Michael and Stuck, Garrett},
  year={2002},
  publisher={Cambridge university press}
}

@article{duarte2006combinatorial,
  title={Combinatorial stability of non-deterministic systems},
  author={Duarte, Pedro and Torres, Maria Joana},
  journal={Ergodic Theory and Dynamical Systems},
  volume={26},
  number={1},
  pages={93--128},
  year={2006},
  publisher={Cambridge University Press}
}

@incollection{duarte2015stability,
  title={Stability of non-deterministic systems},
  author={Duarte, Pedro and Torres, Maria Joana},
  booktitle={From Particle Systems to Partial Differential Equations II: Particle Systems and PDEs II, Braga, Portugal, December 2013},
  pages={193--207},
  year={2015},
  publisher={Springer}
}

@article{peixoto1962structural,
  title={Structural stability on two-dimensional manifolds},
  author={Peixoto, M},
  journal={Topology},
  volume={1},
  pages={101--120},
  year={1962}
}

@article{mane1987proof,
  title={A proof of the $C^{1}$ stability conjecture},
  author={Ma{\~n}{\'e}, Ricardo},
  journal={Publications Math{\'e}matiques de l'IH{\'E}S},
  volume={66},
  pages={161--210},
  year={1987}
}

@article{manning1974there,
  title={There are no new Anosov diffeomorphisms on tori},
  author={Manning, Anthony},
  journal={American Journal of Mathematics},
  volume={96},
  number={3},
  pages={422--429},
  year={1974},
  publisher={JSTOR}
}

@article{BMelo15,
 ISSN = {00029939, 10886826},
 author = {Bernardo Carvalho},
 journal = {Proceedings of the American Mathematical Society},
 number = {2},
 pages = {657--666},
 publisher = {American Mathematical Society},
 title = {HYPERBOLICITY, TRANSITIVITY AND THE TWO-SIDED LIMIT SHADOWING PROPERTY},
 urldate = {2025-08-19},
 volume = {143},
 year = {2015}
}

@article{Micena22,
title = {Some sufficient conditions for transitivity of Anosov diffeomorphisms},
journal = {Journal of Mathematical Analysis and Applications},
volume = {515},
number = {2},
pages = {126433},
year = {2022},
issn = {0022-247X},
doi = {https://doi.org/10.1016/j.jmaa.2022.126433},
url = {https://www.sciencedirect.com/science/article/pii/S0022247X22004474},
author = {Fernando Micena}
}

@article{MicenaHertz22,
author = {Micena, Fernando and Hertz, Jana},
year = {2022},
month = {12},
pages = {1-9},
title = {A relation between entropy and transitivity of Anosov diffeomorphisms},
volume = {29},
journal = {Journal of Dynamical and Control Systems},
doi = {10.1007/s10883-022-09620-2}
}

@article{Salgado25,
      title={On sufficient conditions for the transitivity of homeomorphisms}, 
      author={Maria Carvalho and Vinícius Coelho and Luciana Salgado},
      year={2025},
      eprint={2410.12958},
      archivePrefix={arXiv},
      primaryClass={math.DS},
      journal={arXiv preprint arXiv:2410.12958},
      url={https://arxiv.org/abs/2410.12958}, 
}

@article{BMelo25,
      title={Transitivity of real Anosov diffeomorphisms}, 
      author={Bernardo Carvalho},
      year={2025},
      eprint={2410.15740},
      archivePrefix={arXiv},
      primaryClass={math.DS},
      journal={arXiv preprint arXiv:2410.15740},
      url={https://arxiv.org/abs/2410.15740}, 
}

@InProceedings{FrankWilliams80,
author="Franks, John
and Williams, Bob",
editor="Nitecki, Zbigniew
and Robinson, Clark",
title="Anomalous anosov flows",
booktitle="Global Theory of Dynamical Systems",
year="1980",
publisher="Springer Berlin Heidelberg",
address="Berlin, Heidelberg",
pages="158--174",
isbn="978-3-540-38312-3"
}

\end{document}